\documentclass[11pt]{amsart}
\usepackage{amsmath, amsthm, amssymb}
\usepackage{epsfig}
\newcommand{\E}{\mathbb{E}}
\newcommand{\N}{\mathbb{N}}
\newcommand{\R}{\mathbb{R}}
\newcommand{\Sphere}{\mathbb{S}}
\newcommand{\n}[1]{#1}
\renewcommand\d{{\mathrm d}}
\newcommand\e{{\mathrm e}}
\newcommand\im{{\mathrm i}}
\newcommand\1{{\mathbf 1}}
\newcommand{\eqd}{\stackrel{d}{=}}
\DeclareMathOperator{\sign}{sign}
\newtheorem{theorem}{Theorem}[section]
\newtheorem{lemma}[theorem]{Lemma}
\newtheorem{proposition}[theorem]{Proposition}
\newtheorem{remark}[theorem]{Remark}

\begin{document}
\title{Fractional absolute moments of heavy tailed distributions}
\today
\author[M. Matsui]{Muneya Matsui} 
\address{Department of Business Administration, Nanzan University, 18
Yamazato-cho, Showa-ku, Nagoya 466-8673, Japan.}
\email{mmuneya@nanzan-u.ac.jp}
\author[Z. Pawlas]{Zbyn\v ek Pawlas}
\address{Department of Probability and Mathematical Statistics, Faculty
of Mathematics and Physics, Charles University, 186 75 Prague 8, Czech Republic}
\email{zbynek.pawlas@mff.cuni.cz}

\begin{abstract}
 Several convenient methods for calculation of fractional absolute moments are given with application
 to heavy tailed distributions. We use techniques of fractional differentiation to obtain
 formulae for $\E[|X-\mu|^\gamma]$ with $1<\gamma<2$ and $\mu\in\R$, in terms of Laplace transform or 
 characteristic function.
 The main focus is on heavy tailed distributions, several examples are given with
 analytical expressions of fractional absolute moments.
 As applications, we calculate the fractional moment errors for both prediction and parameter
 estimation problems.  \vspace{2mm} \\
{\it Keywords:} fractional absolute moments, fractional derivatives, heavy tailed distributions, 
characteristic functions, infinitely divisible distributions. \vspace{2mm}\\
2010 Mathematics Subject Classification: Primary 60E10, Secondary 60E07; 62E15.
\end{abstract}

\thanks{
Muneya Matsui's research is partly supported by the JSPS Grant-in-Aid for Research
Activity start-up Grant Number 23800065.
}

\maketitle
\section{Fractional moments} \label{sec:intro}
The purpose of this paper is to study the evaluation tools for goodness of
predictors and estimators which are of infinite variance. This is done
by investigation of the fractional absolute moments. For their calculation
there exist several methods which are not always convenient to use.

Heavy tailed distributions and stochastic processes with infinite variance have found applications 
in many diverse areas (see e.g. Adler et al. \cite{adler:feldman:taqqu:1998}). 
Various statistical methods for these models have been investigated so far. 
Among them, the prediction problems 
have occupied an important place.
To name a few contributions, Hardin et al.
\cite{hardin:samorodnitsky:taqqu:1991} and Samorodnitsky and Taqqu  
\cite{samorodnitsky:taqqu:1991} studied
the conditional expectation for stable random vectors,
i.e.~the best predictor in the sense of minimizing mean squared error if it exists.
In a recent paper, Matsui and Mikosch \cite{matsui:mikosch:2010} 
obtained the conditional expectations for Poisson cluster
models with possibly infinite variance.
The linear predictors for time series models have been 
considered in e.g. Cline and Brockwell \cite{cline:brockwell:1985}
and Kokoszka \cite{kokoszka:1996}. The regression type estimators have also been 
studied in e.g. Blattberg and Sargent \cite{blattberg:sargent:1971} or
Samorodnitsky and Taqqu \cite[Sec. 4]{Samorodnitsky:Taqqu:1994}.

Although, there have been plenty of papers dealing with predictors, we find that 
little attention has been given to the measures of prediction errors and the methods
of their calculation. In \cite{cline:brockwell:1985} and \cite{kokoszka:1996} a certain 
dispersion measure has been proposed, but it is specially intended for the time
series and thus not quite general. 
The problem is that when concerning random elements with no finite
second moments, we can not apply the $L^2$ loss function, which is the most
popular measure because it is easily tractable and intuitively
clear. Therefore, the alternative measures are required.

In this paper we adopt the $L^p$ loss function with $0<p<2$ since we
think it is a natural plausible candidate to evaluate the goodness of prediction or
parameter estimation. Thus, we study the fractional absolute moments $m_p := \E[|X|^p]$ 
of order $0<p<2$. We also consider $\mu$-centered moments $m_{\mu,p} := \E[|X-\mu|^p]$
with $\mu \in \R$. As far as we know, there have not been enough researches 
of the fractional absolute moments except for the special case of first order
absolute moment $m_1$. The reason is that the existing methods are unfamiliar 
or these methods seem to require a lot of numerical work. 

Taking this into consideration, firstly we summarize existing methods
for obtaining the fractional absolute
moments. In particular, we focus on the methods exploiting the Laplace (LP) transform or the 
characteristic function (ch.f.) of the corresponding distribution. It is well-known that the moments 
of integer orders are related to the derivatives of ch.f.~or LP transform at zero. 
More generally, the theory of fractional calculus can be utilized in order to obtain
the non-integer real moments.

There are several works giving the relation between the fractional moments and
the corresponding ch.f~or LP transform. We refer to 
Hsu \cite{hsu:1951}, von Bahr \cite{vonBahr:1965}, Ramachandran \cite{ramachandran:1969},
Brown \cite{brown:1970,brown:1972}, Kawata \cite[Sec. 11.4]{kawata:1972},  
Wolfe \cite{wolfe:1973,wolfe:1975a,wolfe:1975b,wolfe:1978}, Laue \cite{laue:1980,laue:1986},
Zolotarev \cite[Sec. 2.1]{zolotarev:1986}, Paolella \cite[Sec. 8.3]{paolella:2007} and
Pinelis \cite{pinelis:2011}.
The methods using moment generating functions have also been studied, e.g.
by Cressie et al. \cite{cressie:davis:folks:policello:1981} and Cressie and
Borkent \cite{cressie:borkent:1986}. 

Our main tool is the fractional calculus which generalizes ordinary differentiation
and integration to arbitrary order, for details we refer to monographs  
\cite{podlubny:1999} and \cite{samko:1993}. There exist different definitions of fractional
derivatives, we will use the Marchaud fractional derivative. For a complex-valued
function $f$, its fractional derivative of order $\gamma=k+\lambda$ with $k\in
\N$, $0<\lambda<1$, is given by,
see e.g.~\cite[Eq. (2.1)]{laue:1980} or \cite[Sec. 5]{samko:1993}, 
\begin{align*}
\frac{\d^\gamma}{\d t^\gamma}f(t)=\frac{\d^\lambda}{\d t^{\lambda}} f^{(k)}(t)=\frac{\lambda}{\Gamma(1-\lambda)}
\int_{-\infty}^t \frac{f^{(k)}(t)-f^{(k)}(u)}{(t-u)^{1+\lambda}}\,\d u, \quad t \in \R,
\end{align*}
where $f^{(k)}$ is the $k$th derivative of $f$ and $\Gamma$ is the Gamma function. 
We are mostly interested in the fractional absolute moments $m_{1+\lambda}=\E[|X|^{1+\lambda}]$ with
$0<\lambda<1$. For this reason, we will need the fractional derivative of order $1+\lambda$ at zero,
\begin{equation} \label{frac:deriv:1+lambda}
\frac{\d^{1+\lambda}}{\d t^{1+\lambda}}f(t)  \Big |_{t=0}=
\frac{\d^\lambda}{\d t^{\lambda}}f'(t) \Big |_{t=0} =
\frac{\lambda}{\Gamma(1-\lambda)} \int_0^\infty
\frac{f'(0)-f'(-u)}{u^{1+\lambda}}\,\d u.
\end{equation}

The construction of our paper is as follows. 
In the remainder of Section \ref{sec:intro}, we make a brief survey on 
the relation between the fractional absolute moments and fractional
derivatives using references cited above. Several convenient formulae are
also derived. In Section \ref{sec:ID:dist} we
apply the mentioned methods to the infinitely divisible distributions
and examine their fractional absolute moments. Heavy tailed
distributions, such as stable, Pareto, geometric stable and Linnik distributions, are
considered. Especially, in Section \ref{sec:comp:Poisson} we pay attention 
to compound Poisson distribution that is popular in applications.
In the final section, several applications are
presented. The fractional errors of predictions with infinite variance
such as stable distributions, are explicitly calculated. In addition, the estimation errors in
regression models are evaluated by the fractional absolute moments in heavy tailed
cases. 

\subsection{Fractional derivatives of Laplace transforms}

Let $F$ be a distribution function (d.f.) of a non-negative random variable $X$.
Its LP transform is defined as
\begin{align*}
\phi(t) := \int_0^\infty \e^{-tx}\,\d F(x), \quad t \ge 0.
\end{align*}
In \cite[Theorem 1]{wolfe:1975a} the relation between moments of $X$
and the fractional derivative of $\phi$ at zero is given.
We state this result in a slightly modified version. 

\begin{lemma} \label{lem:lp:fractional}
Let $0<\lambda<1$ and let $\phi$ be the LP transform of the
d.f.~$F(x)$ such that $F(x)=0$ for $x<0$. Then $m_{1+\lambda}$ 
exists if and only if $\phi'(0+)$ exists and 
\[
\int_0^\infty \frac{\phi'(u)-\phi'(0+)}{u^{1+\lambda}}\,\d u
\]
exists, in which case 
\[ 
m_{1+\lambda}= 
\frac{\lambda}{\Gamma(1-\lambda)} \int_0^\infty \frac{\phi'(u)-\phi'(0+)}{u^{1+\lambda}}\,\d u.
\]
\end{lemma}

\begin{proof} 
Suppose that $m_{1+\lambda}$ exists, then $\phi'(u)$ exists for all $u > 0$
and is equal to $-\int_0^\infty x\e^{-xu}\,\d F(x)$ and $\phi'(0+)$ exists
and is equal to $-\int_0^\infty x\,\d F(x)$. We use Fubini's theorem to
see that
\begin{align*}
\frac{\lambda}{\Gamma(1-\lambda)} \int_0^\infty \frac{\phi'(u)-\phi'(0+)}{u^{1+\lambda}}\,\d u 
&= \frac{\lambda}{\Gamma(1-\lambda)}\int_0^\infty u^{-(1+\lambda)}
\int_0^\infty x(1-\e^{-xu})\,\d F(x)\,\d u \\
&= \frac{\lambda}{\Gamma(1-\lambda)}\int_0^\infty x^{1+\lambda}
\int_0^\infty \frac{1-\e^{-xu}}{(xu)^{1+\lambda}}x\,\d u\,\d F(x) \\
&= \int_0^\infty x^{1+\lambda}\,\d F(x) <\infty. 
\end{align*}
Conversely, the existence of $\phi'(0+)$ implies 
\[
-\phi'(0+)\ge \int_0^\infty x\e^{-ux}\,\d F(x)=-\phi'(u) 
\]  
for any $u>0$. Hence, the reverse argument yields
$m_{1+\lambda}<\infty$. 
\end{proof}

\begin{remark}
Cressie and Borkent \cite[Theorem 1]{cressie:borkent:1986} show that 
under certain conditions an arbitrary moment of a positive random variable 
is equal to the Caputo fractional derivative $($see \cite[Sec. 2.4.1]{podlubny:1999}$)$
of the corresponding moment generating function at zero.
\end{remark}

\subsection{Fractional derivatives of characteristic functions}

We denote the ch.f.~of a random variable $X$ with
d.f.~$F$ by
\[
\varphi(t) := \int_{-\infty}^\infty \e^{\im tx}\,\d F(x), \quad t \in \R,
\]
and denote that for $X-\mu$ with $\mu \in \R$ by 
\[
\varphi_\mu(t) := \e^{-\im t\mu}\varphi(t), \quad t \in \R.
\]
There are several papers dealing with the relation between
the fractional derivative of $\varphi$ and the fractional
absolute moment. We will work mainly with the result of
Laue \cite{laue:1980} who proved that
\[
m_{2n+\lambda} = \frac{1}{\cos (\frac{\lambda \pi}{2})}
\Re \left[(-1)^n \frac{\d^{2n+\lambda}}{\d t^{2n+\lambda}}
\varphi(t)\Big |_{t=0}\right]
\]
and
\begin{equation} \label{eq:laue}
m_{2n+1+\lambda} = \frac{1}{\sin (\frac{\lambda \pi}{2})}
\Re \left[(-1)^{n+1}\frac{\d^{2n+1+\lambda}}{\d t^{2n+1+\lambda}} \varphi(t)\Big |_{t=0}
\right]
\end{equation}
for any integer $n \geq 0$ and $0<\lambda<1$.
In the following lemma we state the consequences of results
from \cite{laue:1980} and \cite{kawata:1972}.
\begin{lemma} \label{lem:fractional:ch} 
Let $0<\lambda<1$ and let $\varphi$ be the ch.f.~of an arbitrary 
d.f. $F$. 

\begin{itemize}
\item[(1)] $m_{1+\lambda}$ exists if and only if
\begin{equation} \label{eq:lem:fractional:ch:1}
\Re \int_0^\infty \frac{\varphi'(-u)
}{u^{1+\lambda}}\,\d u\;\text{exists,
and}\;\lim_{t\to 0+} \frac{1-\Re\varphi(t)}{t^{1+\lambda}}\;\text{exists.}
\end{equation}
In such a case,   
\begin{align}
\label{eq:lem:fractional:ch:centered}
m_{1+\lambda}= 
\frac{\lambda}{\sin (\frac{\lambda \pi}{2})\Gamma(1-\lambda)}\Re \int_0^\infty
\frac{\varphi'(-u)}{u^{1+\lambda}}\,\d u.  
\end{align}
\item[(2)] A necessary and sufficient condition for the existence of
$m_{1+\lambda}$, $0<\lambda<1$, is that
\begin{align}\label{lem:fractional:ch:kawata:condi}
\Re  \int_0^\infty \frac{1-\varphi(u)}{u^{2+\lambda}}\,\d u<\infty.
\end{align}
In this case,
\begin{align} \label{eq:fractional:ch:kawata}
m_{1+\lambda}=\frac{\lambda
(1+\lambda)}{\sin(\frac{\lambda\pi}{2})\Gamma(1-\lambda)}\Re \int_0^\infty 
\frac{1-\varphi(u)}{u^{2+\lambda}}\,\d u. 
\end{align}

If $\varphi_\mu(t)=\e^{-\im t\mu}\varphi(t)$ satisfies conditions \eqref{eq:lem:fractional:ch:1} 
in (1) or condition \eqref{lem:fractional:ch:kawata:condi} of (2), then
the fractional absolute moment with center $\mu$
$(m_{\mu,1+\lambda}=\E[|X-\mu|^{1+\lambda}])$
is given by 
\begin{align} \label{eq:lem:fractional:ch}
m_{\mu,1+\lambda}=\frac{\lambda}{\sin (\frac{\lambda \pi}{2})\Gamma(1-\lambda)} \Big[
\mu \Im\int_0^\infty \frac{\e^{\im\mu u}\varphi(-u)
}{u^{1+\lambda}}\,\d u +\Re\int_0^\infty \frac{\e^{\im\mu
u}\varphi'(-u)}{u^{1+\lambda}}\,\d u\Big]. 
\end{align}
\end{itemize}
\end{lemma}

\begin{proof}
Part (1) is a special case of \cite[Theorem 2.2(b)]{laue:1980}
and part (2) is contained in \cite[Theorem 11.4.3]{kawata:1972}.
However, for the consistency of the paper and reader's better understanding, we give
the proof which is specific for our parameter ranges $0<\lambda<1$.

A simple calculation yields
\begin{align}
\int_{-\infty}^\infty |x|^{1+\lambda}\,\d F(x) &= \frac{\lambda
(1+\lambda)}{\sin(\frac{\lambda\pi}{2})\Gamma(1-\lambda)}
\int_{-\infty}^\infty |x|^{1+\lambda}\int_0^\infty \frac{1-\cos
u}{u^{2+\lambda}}\,\d u\,\d F(x) \nonumber\\
&=\frac{\lambda
(1+\lambda)}{\sin(\frac{\lambda\pi}{2})\Gamma(1-\lambda)} \int_{-\infty}^\infty \int_0^\infty \frac{1-\cos
ux}{u^{2+\lambda}}\,\d u\,\d F(x) \label{pf:eq1:lem:fractional:ch:kawata:condi}\\
&=\frac{\lambda
(1+\lambda)}{\sin(\frac{\lambda\pi}{2})\Gamma(1-\lambda)} 
\int_0^\infty \frac{1-\Re\varphi(u)}{u^{2+\lambda}}\,\d u, \nonumber
\end{align}
where we use Fubini's theorem in the last step. 
Hence, we obtain the first part of $(2)$. 
Moreover, due to the integration by parts, 
we have for $x\in \R$,

\begin{align} 
\int_0^\infty \frac{1-\cos ux}{u^{2+\lambda}}\,\d u &= \Big[
-\frac{u^{-(1+\lambda)}}{1+\lambda} (1-\cos ux)
\Big]_0^\infty+ \frac{1}{1+\lambda} \int_0^\infty \frac{x\sin
ux}{u^{1+\lambda}}\,\d u \nonumber \\
&= \lim_{u\to 0+} \frac{1}{1+\lambda} \frac{1-\cos
 ux}{u^{1+\lambda}}+\frac{1}{1+\lambda} \int_0^\infty \frac{x\sin
 ux}{u^{1+\lambda}}\,\d u. \label{pf:eq2:lem:fractional:ch:kawata:condi}
\end{align}
If condition \eqref{lem:fractional:ch:kawata:condi} holds, then by the
Lebesgue dominated convergence theorem, 
\[
\int_{-\infty}^\infty \lim_{u\to 0+} \frac{1}{1+\lambda} \frac{1-\cos
 ux}{u^{1+\lambda}}\,\d F(x) = \frac{1}{1+\lambda} \lim_{u\to 0+} 
 \frac{ 1-\Re\varphi(u)}{u^{1+\lambda}} <\infty,
\]
and by Fubini's theorem,
\[
\frac{1}{1+\lambda} \int_{-\infty}^\infty \int_0^\infty \frac{x\sin
 ux}{u^{1+\lambda}}\,\d u\,\d F(x) = \frac{1}{1+\lambda} \int_0^\infty
 \frac{\Re\varphi'(-u)}{u^{1+\lambda}}\,\d u <\infty. 
\]
Thus, conditions \eqref{eq:lem:fractional:ch:1} in (1) 
are satisfied. We can prove the converse in a similar manner, and hence
we showed that conditions in (1) and (2) are equivalent. If
$m_{1+\lambda}<\infty$, then 
\[
\lim_{t\to 0+} \frac{1-\Re \varphi(t)}{t^{1+\lambda}}=0
\]
follows from \eqref{lem:fractional:ch:kawata:condi}. This, together with
 \eqref{pf:eq1:lem:fractional:ch:kawata:condi} and \eqref{pf:eq2:lem:fractional:ch:kawata:condi},
 yields the expression \eqref{eq:lem:fractional:ch:centered}. 

Finally, we substitute the first derivative of the ch.f.~$\varphi_\mu$, which is 
\[
\varphi_\mu'(t) = -\im\mu \e^{-\im t\mu}\varphi(t)+\e^{-\im t\mu}\varphi'(t), 
\quad t \in \R,
\]
into \eqref{eq:lem:fractional:ch:centered} to obtain the desired result \eqref{eq:lem:fractional:ch}.
We may decompose the integral into two parts as in \eqref{eq:lem:fractional:ch}
because the existence of the integral
\[
\Im\int_0^\infty \frac{\e^{\im\mu u}\varphi(-u)
}{u^{1+\lambda}}\,\d u = \int_0^\infty \frac{\cos{\mu
u}\,\Im\varphi(-u)+\sin\mu u\, \Re\varphi(-u)}{u^{1+\lambda}}\,\d u
\]
follows from $|\Re \varphi(-u)| \le 1$ and 
\[
|\Im \varphi(-u)| \le \int_{-\infty}^\infty |\sin (-ux)|\,\d F(x)
\wedge 1 \le \int_{-\infty}^\infty |ux|\,\d F(x) \wedge 1.  
\] 
\end{proof}

\begin{remark}\label{rem:frac:mom}
$($a$)$ Equation \eqref{eq:lem:fractional:ch:centered} follows from
\eqref{frac:deriv:1+lambda} and \eqref{eq:laue} with $n=0$ by noticing that
$\Re\varphi'(0)=0$. \\
$($b$)$ Equation \eqref{eq:fractional:ch:kawata} can also be found as $(2.1.9)$ in
\cite{zolotarev:1986} or $(8.30)$ in \cite{paolella:2007}, in both cases
with differently written constant in front of the integral and with a typo
contained.\\
$($c$)$ Although we will mainly use expressions
\eqref{eq:lem:fractional:ch:centered} and \eqref{eq:lem:fractional:ch},
expression \eqref{eq:fractional:ch:kawata} may be also useful
in some purposes.
\end{remark}
Moreover, Kawata \cite[Theorem 11.4.4]{kawata:1972} has obtained 
expressions for $m_\gamma$, $\gamma>2$, in the form of 
\begin{align*}
m_\gamma = C_\ell \int_0^\infty u^{-(1+\lambda)}\Big[
1-\Re \varphi(u)+\sum_{k=1}^\ell \frac{u^{2k}}{(2k)!} \varphi^{(2k)}(0)
\Big]\,\d u,
\end{align*}
where $\ell\in \N$ is such that $2\ell<\gamma < 2\ell+2$ and $C_\ell$ is a positive constant
depending on $\ell$. 
In other context, 
Wolfe \cite{wolfe:1975a} has derived different technique for calculating
moments $m_\gamma$ of any real order $\gamma \in \R$ from the fractional
derivatives of the ch.f.
Recently, Pinelis \cite{pinelis:2011} has obtained integral expressions of positive-part moments $\E[X_+^p]$ with
$p>0$ in terms of the ch.f. His method is to apply the Fourier-Laplace transform and the Cauchy integral theorem, 
which is different from the fractional derivative approach.

\section{Infinitely Divisible Distributions}
\label{sec:ID:dist}
In this section, we examine the class of infinitely divisible (ID for short) distributions,
whose general definitions and many distributional properties are given by their
ch.f. Many well-known distributions belong to this class and there are magnitude of
applications in different areas (finance, insurance, physics, astronomy etc.).
Here we work on the distribution without Gaussian part, its ch.f.~is 
\begin{align}\label{def:id-ch-f}
\varphi(t)= \exp\Big\{\im\delta t +
\int_{\R}(\e^{\im tx}-1-\im tx\1_{\{x\le 1\}})\,\nu(\d x)\Big\}, 
\quad t \in \R,
\end{align}
where $\delta\in\R$ is a centering constant and $\nu$ is the L\'evy measure
satisfying $\nu(\{0\})=0$ and $\int_{\R}(|x|^2\wedge 1)\,\nu(\d x) < \infty$. 
For more details on the definition and properties, we refer to Sato \cite{sato:1999}.

Although we can not calculate $m_{1+\lambda}$ from 
density functions, because they are not available for most ID distributions, 
we can directly apply the fractional derivative to ch.f.~and obtain
fractional absolute moments.   
An advantage is that we can check the existence of fractional moments by 
the L\'evy measure of ID distributions and we do not need to check conditions  
of Lemma \ref{lem:fractional:ch}.  
The following result is a well-known criterion for moments
(see e.g.~\cite[Corollary 25.8]{sato:1999} or \cite[Theorem 2]{wolfe:1971}).
In our case of interest $\E[|X|^{1+\lambda}]$, $0<\lambda<1$, we find a simple
proof and give it in Appendix. 
\begin{lemma} \label{lem:condition:moment:id}
Let $X$ be an ID distribution with L\'evy measure
$\nu$. Then for $0<\lambda<1$, $m_{1+\lambda}<\infty$ if and only if 
$$
\int_{|x|>1} |x|^{1+\lambda} \,\nu(\d x)<\infty.
$$
\end{lemma}
In what follows, we present the examples.

\subsection{Stable distributions}
As a representative of heavy tailed distributions we firstly consider
stable distributions. A random variable $X$ has a stable distribution with
parameters $0 < \alpha \le 2$, $\sigma \ge 0$, $-1 \le \beta \le 1$
and $\delta \in \R$ if its ch.f.~has the form,
cf.~\cite[Definition 1.1.6]{Samorodnitsky:Taqqu:1994},
\begin{align} \label{ch:stable:general:onedim}
\varphi(t) = \exp\big\{\im\delta t - \sigma^\alpha|t|^\alpha \omega(t)\big\}, \quad t \in \R,
\end{align}

where
\begin{equation} \label{eq:omega}
\omega(t) = \begin{cases} 1 - \im\beta\tan\tfrac{\pi\alpha}{2}\sign(t), & \text{if }\alpha \ne 1 \\ 1 +
	    \im\beta\tfrac{2}{\pi}\sign(t) \log|t|, & \text{if }\alpha = 1. \end{cases}
\end{equation} 
It is well-known that if $\gamma < \alpha < 2$,
the moment of order $\gamma$ exists, otherwise
it does not exist, see e.g. \cite[Sec. 4]{ramachandran:1969} or
\cite[Property 1.2.16]{Samorodnitsky:Taqqu:1994}.
We briefly review the existing results on the moments. 
If $0<\alpha<1$ and $X$ is a stable subordinator with
the LP transform given by $\E[\e^{-tX}]=\exp\{-\sigma^\alpha t^\alpha\}$, then for $-\infty<\gamma<\alpha$,
\[
\E[X^\gamma]=\frac{\Gamma(1-\gamma/\alpha)}{\Gamma(1-\gamma)}\,\sigma^{\gamma},
\]
which is shown by Wolfe \cite[Sec. 4]{wolfe:1975a} or 
Shanbhag and Sreehari \cite{shanbhag:sreehari:1977}.
In symmetric case ($\beta=0$) with $\delta=0$, 
it is shown in \cite[Theorem 3]{shanbhag:sreehari:1977} 
that
\begin{equation}
m_{\gamma}=
\frac{2^{\gamma}\Gamma\bigl((1+\gamma)/2\bigr)\Gamma(1-\gamma/\alpha)}{\Gamma(1-\gamma/2)\Gamma(1/2)}
\,\sigma^{\gamma}, \quad -1<\gamma<\alpha, 
\label{fractional:mom:sym:stable:ss}
\end{equation}
where the authors rely on the decomposition of the symmetric stable
distribution (see also Section 25 in \cite{sato:1999}).  
For general $\beta$ and $\delta=0$, the following relation is proved
by two different methods in Section 8.3 of \cite{paolella:2007}, see also
\cite[p. 18]{Samorodnitsky:Taqqu:1994},
\begin{align} \label{eq:stable:moment:paolella}
m_\gamma = \kappa^{-1}\Gamma\left(1-\frac{\gamma}{\alpha}\right)(1+\theta^2)^\frac{\gamma}{2\alpha}
\cos\left(\frac{\gamma}{\alpha}\arctan \theta\right)\sigma^\gamma,
\quad -1<\gamma<\alpha,
\end{align}
where $\theta = \beta\tan\frac{\pi\alpha}{2}$ and
\[
\kappa = \begin{cases} \Gamma(1-\gamma) 
\cos \frac{\gamma\pi}{2}, & \text{if } \gamma \ne 1, \\
\frac{\pi}{2}, & \text{if } \gamma=1.
\end{cases}
\]
Using the fractional derivative, we obtain from Lemma \ref{lem:fractional:ch}
not only another proof of \eqref{eq:stable:moment:paolella} but also formulae for fractional absolute 
$\mu$-centered moments which seem to be new. 

\begin{proposition} \label{prop:stable:moment}
Let $X$ have a stable distribution with real parameters $\alpha>1$, $|\beta| \le
1$, $\delta=0$ and $\sigma>0$. Then, for $0 < \lambda < \alpha-1$, we have 
\begin{align}
m_{1+\lambda} &=
\frac{\lambda \Gamma\big( 1-\frac{1+\lambda}{\alpha} \big) }{\sin (\frac{\lambda \pi}{2}) \Gamma(1-\lambda)} 
\,\sigma^{1+\lambda}(1+\theta^2)^{\frac{1+\lambda}{2\alpha}-\frac{1}{2}} 
\label{eq:general:stable:} \\
&\quad \times
\Big\{
\cos\Big[
\big( 1-\frac{1+\lambda}{\alpha} \big) \arctan\theta\Big]+\theta \sin \Big[
\big( 1-\frac{1+\lambda}{\alpha} \big) \arctan\theta\Big]\Big\}, \nonumber
\end{align}
and for $\mu\in \R$, 
\begin{align}
&m_{\mu,1+\lambda} = \label{eq:mu:general:stable:}
\frac{\lambda}{\sin(\frac{\lambda \pi}{2})\Gamma(1-\lambda)}
\Big\{
\mu \int_0^\infty u^{-(1+\lambda)} \e^{-\sigma^\alpha u^\alpha} \sin
 \big(
\mu u-\theta\sigma^\alpha u^\alpha
\big)\,\d u \\
&\quad +\alpha \sigma^\alpha \int_0^\infty
 u^{\alpha-\lambda-2}\e^{-\sigma^\alpha u^\alpha} \Big[\cos
 \big(\mu u-\theta \sigma^\alpha u^\alpha\big)
 -\theta \sin \big(\mu u-\theta \sigma^\alpha u^\alpha\big)\Big]
 \,\d u\Big\}, \nonumber
\end{align}
where $\theta = \beta\tan\frac{\pi\alpha}{2}$.
If $X$ is symmetric $(\beta=0)$, it follows that 
\begin{align}
m_{1+\lambda} = 
\frac{\lambda\Gamma(1-\frac{1+\lambda}{\alpha})}{\sin(\frac{\lambda\pi}{2}) \Gamma(1-\lambda)}
\,\sigma^{1+\lambda}
\label{fractional:mom:sym:stable}
\end{align}
and
\begin{align} 
m_{\mu,1+\lambda} &= \frac{\lambda\,
 \sigma^{1+\lambda}}{\sin(\frac{\lambda\pi}{2})\Gamma(1-\lambda)} \Big[
\frac{\mu}{\sigma} \int_0^\infty u^{-(1+\lambda)} \e^{-u^\alpha} \sin
 \big(
\frac{\mu u}{\sigma} 
\big)\,\d u \label{eq:mu:general:stable:symm}
\\
&\hspace{3cm} + \alpha\int_0^\infty u^{\alpha-\lambda-2} \e^{-u^\alpha} \cos\big(
\frac{\mu u}{\sigma}\big)\,\d u\Big]. \nonumber
\end{align}
\end{proposition}


\begin{proof} 

We begin with the expression of $m_{\mu,1+\lambda}$. 
Let $\varphi$ be the ch.f.~of a stable distribution
with $\delta=0$ and $\alpha>1$. Since we have, for $u>0$, 
\begin{align*}
\Im\e^{\im\mu u}\varphi(-u) &= \exp\{-\sigma^\alpha u^\alpha\}
\sin \big(\mu u- \theta \sigma^\alpha u^\alpha\big), \\
\Re\e^{\im\mu u}\varphi'(-u) &= \alpha \sigma^\alpha u^{\alpha-1}
\exp\{-\sigma^\alpha u^\alpha\} \cos \big(\mu u- \theta\sigma^\alpha u^\alpha\big) \\
&\quad - \alpha \theta \sigma^\alpha u^{\alpha-1}
\exp\{-\sigma^\alpha u^\alpha\} \sin \big(\mu u-\theta \sigma^\alpha u^\alpha
\big), 
\end{align*}
inserting these into \eqref{eq:lem:fractional:ch} of Lemma
\ref{lem:fractional:ch}, we get \eqref{eq:mu:general:stable:}.
For $m_{1+\lambda}$, we let $\mu=0$
in \eqref{eq:mu:general:stable:} and use change of variables theorem to obtain
\begin{align*}
m_{1+\lambda} =\frac{\lambda\,\sigma^{1+\lambda}}{\sin(\frac{\lambda \pi}{2})\Gamma(1-\lambda)} 
\Big(\int_0^\infty u^{-\frac{1+\lambda}{\alpha}} \e^{-u}\cos\theta u
\,\d u + \theta \int_0^\infty
u^{-\frac{1+\lambda}{\alpha}} \e^{-u} \sin \theta u\,\d u
\Big). 
\end{align*}
Now applying the formulae 
(3.944-5) and (3.944-6) in \cite[p.~498]{Gradshteyn:Ryzhik:2007},
we get \eqref{eq:general:stable:}. Finally, letting $\beta=0$ and applying change of variables,
the symmetric case is obtained.  
\end{proof}

After some manipulation one can show that 
\eqref{fractional:mom:sym:stable} coincides with
\eqref{fractional:mom:sym:stable:ss}
and \eqref{eq:general:stable:} 
coincides with \eqref{eq:stable:moment:paolella}
for $\gamma=1+\lambda$.
Figure \ref{fig:1} shows the fractional absolute moments with center $\mu$,
computed numerically from the representation \eqref{eq:mu:general:stable:}.
We remark that even if this representation includes some integral expressions, 
it would be useful since most stable distributions have no explicit density functions.

\begin{figure}
\begin{center}
\includegraphics[width=0.475\textwidth]{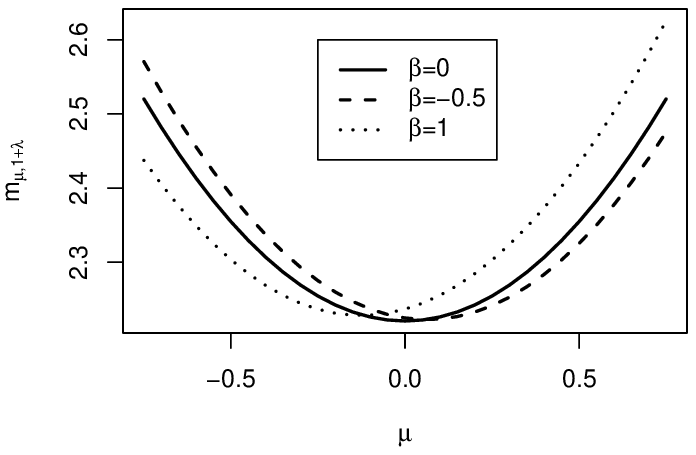}
\includegraphics[width=0.475\textwidth]{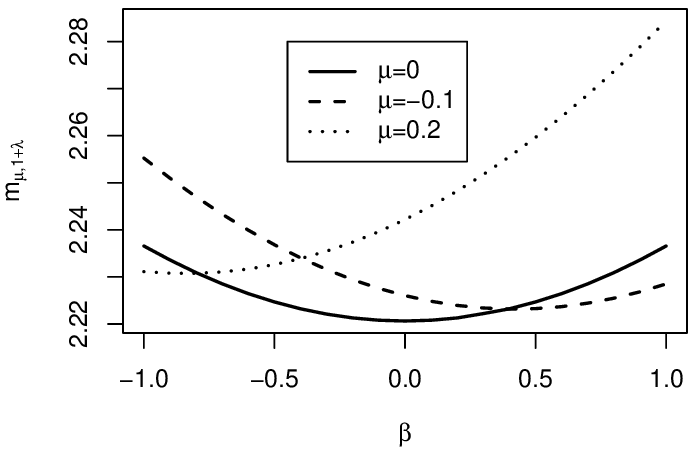}
\end{center}
\caption{The moments $m_{\mu,1+\lambda}$ of stable distribution with parameters
$\alpha=1.8$, $\beta \in [-1,1]$, $\delta=0$ and $\sigma=1$. We choose $\lambda=0.5$
and depict the dependence on $\mu$ for three choices of $\beta$ (left) and
the dependence on $\beta$ for three choices of $\mu$ (right).}
\label{fig:1}
\end{figure}

%

\subsection{Pareto law}
Another heavy tailed distribution is the Pareto distribution which
has density and ch.f.~given by 
\begin{align*}
f(x) &= \alpha (1+x)^{-\alpha-1},\quad x>0, \\
\varphi(t) &= \alpha \int_0^\infty \e^{\im ty} (1+y)^{-\alpha-1}\,\d y, \quad t \in \R,
\end{align*}
respectively, with real positive parameter $\alpha >0$.
This distribution belongs to ID distribution (see Remark 8.12 in \cite{sato:1999}).
The fractional absolute moment $m_{1+\lambda}$ exists if and only if $1+\lambda<\alpha$. 
Though the density function is explicit, we obtain $m_{\mu,1+\lambda}$
from the fractional derivative of ch.f. 
Using \eqref{eq:lem:fractional:ch} of Lemma \ref{lem:fractional:ch}, we have, for $1 < 1+\lambda < \alpha$,
\[
m_{\mu,1+\lambda} =\alpha \Big[
(\mu+1)^{1+\lambda-\alpha}B(\alpha-1-\lambda,2+\lambda) + \frac{\mu^{2+\lambda}}{2+\lambda}\,{}_2F_1(1,\alpha+1,3+\lambda;-\mu)
\Big],
\]
where $B$ is the beta function and ${}_2F_1$ is the Gauss hypergeometric function.

Although the following examples are not always in ID
distributions, they are closely related and could be heavy tailed.

\subsection{Geometric stable law}
A geometric stable distribution has similar properties to the stable 
distribution. The ch.f.~is given as
\[
\varphi(t) = \bigl[1+\sigma^{\alpha}|t|^{\alpha} \omega(t) - \im \delta t\bigr]^{-1},  \quad t \in \R,
\]
where $0<\alpha<2$, $\delta \in \R$ and $\omega(t)$ is defined by \eqref{eq:omega}.
However, its density function has no analytical expression. The tail
behavior is the same as that of stable distribution, see e.g. Kozubowski
et al. \cite{kozubowski:podgorski:samorodnitsky:1999}. 
\begin{lemma} \label{lem:geometric:stable}
Let $X$ has a geometric stable distribution with $\delta=0$.
Then, for $1 < 1+\lambda < \alpha$ and $\mu \in \R$,
\begin{align*}
m_{\mu,1+\lambda} & =
\frac{\lambda}{\sin(\frac{\lambda\pi}{2})\Gamma(1-\lambda)}
\Bigg\{
\mu \int_0^\infty u^{-(1+\lambda)} \frac{(1+\sigma^\alpha
u^\alpha)\sin\mu u-\theta \sigma^\alpha u^\alpha\cos\mu u}{(1+\sigma^\alpha
u^\alpha)^2+(\theta \sigma^\alpha u^\alpha)^2}\,\d u \\
&\quad - \alpha \sigma^\alpha \int_0^\infty 
u^{\alpha-\lambda-2}\,
\frac{\cos\mu u+\theta\sin \mu u }{
(1+\sigma^\alpha u^\alpha)^2+(\theta \sigma^\alpha u^\alpha)^2
}\,\d u + 2 \alpha\sigma^\alpha \int_0^\infty u^{\alpha-\lambda-2} \\
&\quad \times \frac{\bigl[
(1+\sigma^\alpha u^\alpha)\cos\mu u+\theta \sigma^\alpha u^\alpha \sin\mu u\bigr]
(1+\sigma^\alpha u^\alpha + \theta^2\sigma^\alpha u^\alpha)}{
\bigl[(1+\sigma^\alpha u^\alpha)^2+(\theta \sigma^\alpha u^\alpha)^2\bigr]^2
}\,\d u 
\Bigg\}
\end{align*} 
and 
\[
m_{1+\lambda} = \frac{\lambda\,\sigma^{1+\lambda}}{\sin(\frac{\lambda
\pi}{2})\Gamma(1-\lambda)} \int_0^\infty v^{-\frac{1+\lambda}{\alpha}}
\frac{(1+v)^2+(\theta v)^2+2\theta^2 v}{\bigl[(1+v)^2+(\theta v)^2\bigr]^2}\,\d v, 
\] 
where $\theta=\beta\tan\tfrac{\pi\alpha}{2}$. 
\end{lemma} 
\begin{proof}
The proof is a direct application of Lemma \ref{lem:fractional:ch}.
\end{proof}

If we put $\theta=0$ the results coincide with the standard Linnik law case. 

\subsection{Linnik law} \label{subsec:linnik}
We consider a version of Linnik distribution given by Linnik \cite{linnik:1953}. Its  
density function is not explicit, while its ch.f.~has the form
\[
\varphi(t)=(1+\sigma^\alpha|t|^\alpha)^{-\beta}, \quad t \in \R,
\]
where $0< \alpha \le 2$ is the stability parameter, $\sigma>0$ is the scale parameter and $\beta>0$.  
By the method of fractional derivative, 
we recover the result of Lin \cite{lin:1998} as 
\[
m_{1+\lambda} =\frac{\lambda \beta \sigma^{1+\lambda}}{\sin
(\frac{\lambda\pi}{2})\Gamma(1-\lambda)}
B\Big(1-\frac{1+\lambda}{\alpha},\beta+\frac{1+\lambda}{\alpha}\Big),  
\] 
where $1<1+\lambda < \alpha$.
The fractional absolute moment of order $1<1+\lambda< \alpha$ with center $\mu\in \R$ is 
\begin{equation}
m_{\mu,1+\lambda} = \frac{\lambda \sigma^{1+\lambda}}{\sin
(\frac{\lambda\pi}{2})\Gamma(1-\lambda)} \Big[\frac{\mu}{\sigma}
\int_0^\infty \frac{ u^{-(1+\lambda)} \sin \frac{\mu u}{\sigma}}{(1+u^\alpha)^\beta}\,\d u+\alpha\beta
\int_0^\infty 
\frac{u^{\alpha-\lambda-2} \cos \frac{\mu u}{\sigma} }{(1+u^\alpha)^{\beta+1}}\,\d u
\Big]. 
\label{eq:fracmom:linnik}
\end{equation}
For $\beta=1$ these equations coincide with those in Lemma \ref{lem:geometric:stable}
for $\theta=0$.

\subsection{Combination of stable law and Linnik law} 

Since
\[
 \lim_{\beta \to
 \infty}(1+\sigma^\alpha|t|^{\alpha}/\beta)^{-\beta}=\e^{-\sigma^\alpha|t|^\alpha},\quad
 t\in\R,\quad 0<\alpha \le 2,
\]
a symmetric stable distribution is a limit of Linnik-type distributions.
We consider their combination keeping both exponents $\alpha$ to be
identical. Let $X$ be a symmetric stable random variable with
ch.f.~$\varphi(t)=\e^{-|t|^\alpha}$ and let $Y$ be a random variable with
Linnik-type distribution and
ch.f.~$\varphi(t)=(1+|t|^\alpha/\beta)^{-\beta}$. Then we may express
$\E[|X-Y|^{1+\lambda}]$ by taking expectation of 
\eqref{eq:fracmom:linnik} with $\mu$ replaced by $X$ and
$\sigma=\beta^{-1/\alpha}$. As the result we obtain 
\begin{align*}
 \E[|X-Y|^{1+\lambda}] 
&= \frac{\lambda\beta^{1-\frac{1+\lambda}{\alpha}}}{\sin(\frac{\lambda\pi}{2})\Gamma(1-\lambda)} 
\Big[
\int_0^\infty u^{-\frac{1+\lambda}{\alpha}}(1+u)^{-\beta}\e^{-\beta
 u}\,\d u \\
 &\hspace{3.5cm} + \int_0^\infty u^{-\frac{1+\lambda}{\alpha}}(1+u)^{-\beta-1}\e^{-\beta u}\,\d u
\Big] \\
&= \frac{\lambda
 \beta^{1-\frac{1+\lambda}{\alpha}}\Gamma(1-\frac{1+\lambda}{\alpha})}{\sin(\frac{\lambda\pi}{2})\Gamma(1-\lambda)} 
\Big[ 
U\big(1-\frac{1+\lambda}{\alpha},2-\beta-\frac{1+\lambda}{\alpha};\beta
 \big) \\
&\hspace{3.5cm} + U\big(1-\frac{1+\lambda}{\alpha},1-\beta-\frac{1+\lambda}{\alpha};\beta\big)
\Big],
\end{align*}
where $U$ is the confluent hypergeometric function \cite[(9.210-2)]{Gradshteyn:Ryzhik:2007}.

\subsection{Subordinator} 

For practical reasons, it is desirable to express the moments $m_{1+\lambda}$ through
L\'evy measure $\nu$ since ID distributions without Gaussian part are completely
characterized by centering parameter $\delta$ and L\'evy
measure. However, in the light of \eqref{def:id-ch-f}, such expressions
seem to be too formal and too complicated, thus they seem to be not very useful.
Here, we confine our interest to some well-known distributions. However,
for small classes of ID distributions general expressions of
$m_{1+\lambda}$ by $\nu$ are worth considering. 
We pick out the class of 
subordinator $($positive valued ID distributions$)$ and 
that of compound Poisson distributions, the latter is treated
in Section \ref{sec:comp:Poisson}. 

For subordinator,
we apply Lemma \ref{lem:lp:fractional} and obtain a relatively simple
expression. The LP transform of a subordinator can be found in \cite[Theorem 30.1]{sato:1999}. 
\begin{proposition}
Let $X$ be a positive valued ID random variable with
shift parameter $\delta \ge 0$ and 
L\'evy measure $\nu$ such that
$\int_{(0,\infty)}(1\wedge |s|)\,\nu(\d s)<\infty$.
The LP transform is given by  
\[
\phi(t) = \e^{\Psi(-t)}, \quad t \geq 0,
\]
where 
\[
\Psi(t)=\delta t+\int_{(0,\infty)}(\e^{st}-1)\,\nu(\d s).
\]
Then it follows that
\[
m_{1+\lambda}= \frac{\lambda}{\Gamma(1-\lambda)}\Big[
\delta \int_0^\infty \frac{1-\e^{\Psi(-u)}}{u^{1+\lambda}}\,\d u + \int_0^\infty
s\,\nu(\d s) \int_0^\infty \frac{1-\e^{-us}\e^{\Psi(-u)}}{u^{1+\lambda}}\,\d u
\Big]. 
\]
\end{proposition}

\section{Compound Poisson distribution}\label{sec:comp:Poisson}
Among ID distributions we focus on compound Poisson (CP for short)
distribution which can easily manage the tail behavior by assuming a heavy tailed jump
distribution. However, since most distributions do not have explicit
representations, we rely on the ch.f.~or the LP transform
for calculating fractional moments. 
Let $c$ be the intensity parameter of underlying Poisson distribution and
$\nu$ jump measure. The CP distribution has the
following ch.f.
\begin{align}\label{eq:ch:CP}
\varphi(t)=\exp\left\{c\int (\e^{\im tx}-1)\,\nu(\d x)\right\}:=\exp\bigl\{c(\varphi_J(t)-1)\bigr\},
\quad t \in \R,
\end{align}
where $\varphi_J(t) := \int \e^{\im tx}\,\nu(\d x)$ is the ch.f.~of jump
distribution. If the jump distribution has positive support, we obtain the LP
transform 
\[
\phi(t) = \exp\bigl\{c(\phi_J(t)-1)\bigr\},\quad t \ge 0,
\]
where $\phi_J(t) := \int \e^{-tx}\,\nu(\d x)$.
The fractional absolute moments are expressed in the following lemma. 
The proof is just an application of Lemma \ref{lem:lp:fractional} and Lemma \ref{lem:fractional:ch}.

\begin{lemma} \label{lem:compound:Poi} 
Let $\varphi(t)$ be the ch.f.~of CP given by
\eqref{eq:ch:CP}, then we have the following form for fractional $\mu$-centered 
moments of order $1 < 1+\lambda < 2$,
\begin{align*} 
 m_{\mu,1+\lambda} &= \frac{\lambda}{\sin (\frac{\lambda \pi}{2}) \Gamma(1-\lambda)}
\Big\{ \mu \int_0^\infty h_\lambda(u)\sin[\mu u+c \Im(\varphi_J(-u))]\,\d u\\
&\hspace{2cm} + c\int_0^\infty h_\lambda(u) \Re (\varphi_J'(-u)) \cos[\mu u+ c\Im(\varphi_J(-u))]\,\d u \\
&\hspace{2cm} - c\int_0^\infty h_\lambda(u) \Im (\varphi_J'(-u)) \sin[\mu u+ c\Im(\varphi_J(-u))]\,\d u\Big\}, 
\end{align*}
where $h_\lambda(u) = u^{-(1+\lambda)} \exp\{c[ \Re (\varphi_J(-u))-1]\}$.

If the jump distribution is symmetric, i.e.~$\Im\varphi_J(u)=0$, we
have
\begin{align*}
m_{\mu,1+\lambda} &= \frac{\lambda}{\sin (\frac{\lambda \pi}{2}) \Gamma(1-\lambda)}\Big\{ \mu 
\int_0^\infty u^{-(1+\lambda)} \sin(\mu u)\exp\{
c(\varphi_J(-u)-1)
\}\,\d u\\
&\hspace{2cm} +c \int_0^\infty u^{-(1+\lambda)} \varphi_J'(-u)
\cos(\mu u)\exp\{
c(\varphi_J(-u)-1)
\}\,\d u \Big\}
\end{align*}
and moreover
\[
 m_{1+\lambda}= \frac{\lambda c}{\sin (\frac{\lambda \pi}{2})\Gamma(1-\lambda)} 
\int_0^\infty
u^{-(1+\lambda)} \varphi_J'(-u)\exp\{
 c(\varphi_J(-u)-1)\}\,\d u. 
\]
If the jump distribution has positive support, we have 
\[
m_{1+\lambda} = \frac{\lambda c}{\Gamma(1-\lambda)} \int_0^\infty
u^{-(1+\lambda)} \big[
\phi_J'(u)\exp\{c(\phi_J(u)-1)\}-\phi_J'(0)
\big]\,\d u.
\]
\end{lemma}


In what follows, we will examine jumps given by well known distributions,
which are not always heavy tailed, and try to obtain analytical
expressions. Since they require a lot of numerical integrals and special
functions we 
just mention the key steps of derivation.

\subsection{exponential jump}
The LP transform of the exponential distribution with parameter $\beta$, 
i.e.~with density function $f(x)=\frac{1}{\beta} \e^{x/\beta}$, $x \ge 0$, is 
$\phi_J(t)=1/(1+\beta t)$, $t \ge 0$. Then
due to Lemma \ref{lem:compound:Poi},
fractional absolute moments for $0 < \lambda < 1$ are given by 
\begin{align*}
m_{1+\lambda} &= \frac{\lambda c\beta}{\Gamma(1-\lambda)}  \int_0^\infty \frac{(1+\beta
u)^2-\exp\{c(\frac{1}{1+\beta u}-1)\}}{u^{1+\lambda}(1+\beta u)^2}\,\d u \\
&= c\beta^{1+\lambda}\Gamma(2+\lambda) \Big[
{}_1F_1(1-\lambda;2;-c)+\frac{c}{2}\,{}_1F_1(1-\lambda;3;-c)
\Big],
\end{align*}
where ${}_1F_1$ is the confluent hypergeometric function \cite[(9.210-1)]{Gradshteyn:Ryzhik:2007}
and we use (3.383-1) and (3.191-3) of \cite{Gradshteyn:Ryzhik:2007}.

\subsection{symmetric stable jump}
Recall that the ch.f.~is $\varphi_J(t)=\e^{-|t|^\alpha}$ with
$1<\alpha<2$ and thus we apply Lemma \ref{lem:compound:Poi} with $1 <1+\lambda <\alpha$ to obtain 
the following series representation, 
\begin{align*}
 m_{1+\lambda} &= \frac{\lambda \alpha c}{\sin
(\frac{\lambda \pi}{2})\Gamma(1-\lambda)} \int_0^\infty u^{\alpha-\lambda-2}
\e^{-u^\alpha}\exp\{
c(\e^{-u^\alpha}-1)
\}\,\d u \\
&= \frac{\lambda c}{\sin (\frac{\lambda \pi}{2}) \Gamma(1-\lambda)}
\int_0^\infty v^{-(1+\lambda)/\alpha} \e^{-v} \exp\{c(\e^{-v}-1)\}\,\d v \\
&=\frac{\lambda c}{\sin (\frac{\lambda\pi}{2}) \Gamma(1-\lambda)}\e^{-c} \sum_{n=0}^\infty c^n 
\frac{\Gamma\bigl(1-(1+\lambda)/\alpha\bigr)}{n!(n+1)^{1-(1+\lambda)/\alpha}}.
\end{align*}
Again by Lemma \ref{lem:compound:Poi}, shifted fractional moments $\E[|X-\mu|^{1+\lambda}]$
with $1<1+\lambda <\alpha$ are obtained as 
\begin{align*}
m_{\mu,1+\lambda} &= \frac{\lambda }{\sin (\frac{\lambda \pi}{2}) \Gamma (1-\lambda)} 
\Big[ \mu \int_0^\infty 
u^{-(1+\lambda)}\sin(\mu u) \exp\{c(\e^{-u^\alpha}-1)\}\,\d u \\
& 
\hspace{2cm}+ c \alpha \int_0^\infty u^{\alpha-\lambda-2} \cos(\mu u)
\e^{-u^\alpha }\exp\{c(\e^{-u^\alpha}-1)\}\,\d u  \Big].
\end{align*}

\subsection{Linnik distribution jump}
Let $\varphi_J$ be the ch.f.~of Linnik distribution with parameters $\alpha>1$,
$\beta>0$ and $\sigma=1$. Since 
\[
\varphi_J'(-u) = \alpha \beta (1+u^\alpha)^{-\beta-1}u^{\alpha-1},\quad u>0,
\]
from Lemma \ref{lem:compound:Poi} and change of variables formula 
($v = (1+u^\alpha)^{-\beta}$) it follows that 
\[
m_{1+\lambda} = \frac{\lambda\,c\,\e^{-c}}{\sin
(\frac{\lambda\pi}{2})\Gamma(1-\lambda)} \int_0^1
v^{\frac{1+\lambda}{\alpha\beta}}(1-v^{\frac{1}{\beta}})^{-\frac{1+\lambda}{\alpha}} \e^{cv}\,\d v
\]
for $1 < 1 + \lambda < \alpha$.
If $\beta=1$, the jump distribution is the symmetric geometric stable distribution
and we have
\[
m_{1+\lambda} = \frac{\lambda\,c\,\e^{-c}}{\sin(\frac{\lambda
\pi}{2})\Gamma(1-\lambda)}B\bigl(1-\frac{1+\lambda}{\alpha},1+\frac{1+\lambda}{\alpha}\bigr)
\,{}_1F_1\bigl(1+\frac{1+\lambda}{\alpha};2;c\bigr),
\]
where we use (3.383-1) in \cite{Gradshteyn:Ryzhik:2007}.

\subsection{deterministic jump of size 1 (simple Poisson)}
Substituting its LP transform $\phi_J(t)=\e^{-t}$ into the
expression in Lemma \ref{lem:compound:Poi}, we have 
\[
m_{1+\lambda}= \frac{\lambda\,c\,\e^{-c} }{\Gamma(1-\lambda)}
\int_0^\infty
\frac{\e^{c}-\e^{-u}\e^{c\e^{-u}}}{u^{1+\lambda}}\,\d u,  
\]
which is rewritten by the Taylor expansion as 
\[
m_{1+\lambda} = \frac{\lambda\,c\,\e^{-c}}{\Gamma(1-\lambda)} \sum_{k=0}^\infty \frac{c^k}{k!}
\int_0^\infty \frac{1- \e^{-(k+1)u}}{u^{1+\lambda}}\,\d u = \e^{-c}\sum_{k=0}^\infty
\frac{k^{1+\lambda}c^k}{k!}, 
\]
where the final expression can be directly obtained from the probability mass function. 

\begin{remark}
If the jump distribution has  
reproductive property, i.e.~it is convolution-closed, we have another
method for determining the fractional absolute moments. 
Write the CP random variable as $S_N=\sum_{j=1}^N X_j$, where
$N$ has the Poisson distribution with parameter $c$ and $(X_j)$ is an iid sequence such
that $X_1$ has reproduction property. Denote the ch.f.~of $k$th convolution of $X_1$ by
$\varphi_k(t)$, then under suitable conditions we have 
\begin{align*}
m_{1+\lambda} &= \frac{\lambda}{\sin (\frac{\lambda\pi}{2})
\Gamma(1-\lambda)}\E\Big[ \Re \int_0^\infty
\frac{\varphi_N'(-u)}{u^{1+\lambda}}\,\d u \Big]  \\
&= \frac{\lambda}{\sin (\frac{\lambda \pi}{2})
\Gamma(1-\lambda)} \sum_{k=0}^\infty 
\frac{c^k}{k!}\e^{-c} \Big[\Re \int_0^\infty
\frac{\varphi_k'(-u)}{u^{1+\lambda}}\,\d u \Big]. 
\end{align*}
In case of the LP transform we denote that of $k$th convolution of $X_1$ by
$\phi_k(t)$, $t\ge0$, and from Lemma \ref{lem:lp:fractional} we obtain
\[
m_{1+\lambda} 
= \frac{\lambda\,\e^{-c}}{\Gamma(1-\lambda)} \sum_{k=0}^\infty \frac{c^k}{k!}
\int_0^\infty \frac{\phi_k'(u)-\phi_k'(0+)}{u^{1+\lambda}}\,\d u. 
\] 
\end{remark}

\section{Applications}

\subsection{Evaluation of conditional expectation for stable law}

Conditional expectations of stable random vectors have been intensively 
investigated in \cite{hardin:samorodnitsky:taqqu:1991} and 
\cite{samorodnitsky:taqqu:1991},
since stable laws are often thought as natural generalization of the 
Gaussian random vector for which the minimizer of the mean squared
error given some components of the vector is the conditional expectation. 
However, their evaluations have not been examined enough. 
In what follows, we evaluate the goodness of several
predictors given by conditional expectations through their fractional moments.  

Firstly, we consider general results for a bivariate stable random vector
with ch.f.
\[
\varphi(t_1,t_2):=\E[\e^{\im(t_1X_1+t_2X_2)}], \quad (t_1,t_2)\in\R^2,
\]
which can be written as 
\begin{align} \label{ch:bivariate:stable}
\varphi(t_1,t_2)
&= \exp\Big\{
-\int_{\Sphere^1}|t_1s_1+t_2s_2|^\alpha\Big[1-
\im\tan\frac{\pi\alpha}{2} \sign(t_1s_1+t_2s_2)\Big]\,\Gamma(\d s) \\
&\hspace{1.3cm} +\im(t_1\delta_1+t_2\delta_2)\Big\}, \nonumber
\end{align}
where $\Gamma$ is a finite measure on the unit sphere $\Sphere^1$, called spectral
measure, and we let $\alpha>1$, see \cite[Theorem 2.3.1]{Samorodnitsky:Taqqu:1994}. 
Our aim is to linearly approximate $X_2$ by $X_1$ and evaluate the fractional error 
of order $1<\gamma<2$. The situation includes various settings, e.g. if stable random 
vectors are symmetric, i.e.
\begin{align*} 
\varphi(t_1,t_2)
&= \exp\Big\{
-\int_{\Sphere^1}|t_1s_1+t_2s_2|^\alpha\,\Gamma(\d s)
\Big\},
\end{align*}
then it is proved that $\E[X_2\mid X_1]=c X_1$ with some constant $c$, see \cite[Theorem
4.1.2]{Samorodnitsky:Taqqu:1994} or \cite[Theorem 3.1]{samorodnitsky:taqqu:1991}. For 
general case we refer to \cite[Theorem 3.1]{hardin:samorodnitsky:taqqu:1991}. For convenience, 
we assume $(\delta_1,\delta_2)=\bf 0$, 
the general result for $(\delta_1, \delta_2)\neq \bf 0$ can be obtained
in the same manner. 

\begin{proposition} \label{prop:bivariate:stable:condi:mom}
Let $(X_1,X_2)$ be a bivariate stable random vector defined by
\eqref{ch:bivariate:stable} such that $(\delta_1,\delta_2)=\bf 0$. Then
for any constant $c$ and $1<1+\lambda<\alpha$, it follows that
\begin{align*}
\E[|X_2-cX_1|^{1+\lambda}] &= \frac{\lambda \Gamma( 1-\frac{1+\lambda}{\alpha} ) 
}{\sin (\frac{\lambda \pi}{2}) \Gamma(1-\lambda)} 
\,\sigma_0^{1+\lambda}( 1+\theta_0^2)^{\frac{1+\lambda}{2\alpha}-\frac{1}{2}}
(\cos\psi_\lambda + \theta_0\sin\psi_\lambda)
\end{align*}
where 
\begin{equation}\label{eq:predict:multi:stable:sigmabeta}
\sigma_0=\Big(
\int_{\Sphere^1}|s_2-cs_1|^\alpha\, \Gamma(\d s)
\Big)^{1/\alpha}, \quad 
\beta_0 =\frac{\int_{\Sphere^1}\sign(s_2-cs_1)|s_2-cs_1|^\alpha\,\Gamma(\d s)}{
\int_{\Sphere^1} |s_2-cs_1|^\alpha\,\Gamma(\d s)},
\end{equation}
$\theta_0 = \beta_0 \tan\frac{\pi\alpha}{2}$ and $\psi_\lambda = \big( 1-\frac{1+\lambda}{\alpha} \big) \arctan\theta_0$.
In symmetric case, we have 
\[
\E[|X_2-cX_1|^{1+\lambda}] = \frac{\lambda \Gamma( 1-\frac{1+\lambda}{\alpha} ) 
}{\sin (\frac{\lambda \pi}{2}) \Gamma(1-\lambda)} \Big(
\int_{\Sphere^1}|s_2-cs_1|^\alpha\, \Gamma(\d s)
\Big)^{\frac{1+\lambda}{\alpha}}.
\]
\end{proposition}
\begin{proof} 
The fractional derivative of the ch.f.~of $X_2-cX_1$ is
calculated. We put $t_1=-cu$ and $t_2=u$ in
\eqref{ch:bivariate:stable}, then we regard it as a function of $u$,
\begin{align*}
\E[\e^{\im u(X_2-cX_1)}] 
= \exp\Big\{
-|u|^\alpha \int_{\Sphere^1}|s_2-cs_1|^\alpha\,\Gamma(\d s) \Big(
1-\im\beta_0\tan\frac{\pi\alpha}{2} \sign(u) 
\Big)
\Big\}.
\end{align*}
In view of \eqref{ch:stable:general:onedim}, this is the ch.f.~of one-dimensional
$\alpha$-stable distribution with parameters
$(\beta,\sigma,\delta)=(\beta_0,\sigma_0,0)$. 
Hence, 
we apply Proposition \ref{prop:stable:moment} to obtain the result. 
\end{proof}

\noindent
{\bf Examples.} As examples we consider predictions for two bivariate stable random
vectors and one stable process. First we treat a bivariate
stable random vector considered by \cite[p.~183]{nguyen:1995} such that ch.f.~of 
$(X_1,X_2)$ satisfies for $|a|<1$, $\alpha \neq 1$,

\begin{align*}
\varphi(t_1,t_2) &= \E[\e^{\im(t_1X_1+t_2X_2)}] 
= \exp\Big\{
-\sigma^\alpha|t_2|^\alpha \Big[1+\im\beta 
\tan\frac{\pi\alpha}{2}\sign(t_2) \Big] \\
& \quad 
-\frac{\sigma^\alpha|t_1+at_2|^\alpha }{1-|a|^\alpha} \Big[
1+\im\beta\tan \frac{\pi\alpha}{2} \frac{1-|a|^\alpha}{1-\sign(a)|a|^\alpha}\sign (t_1+at_2)
\Big]
\Big\}.
\end{align*}
The conditional ch.f.~is 
\begin{align*}
\varphi_{X_1=x}(t) := \E[ \e^{\im tX_2}\mid X_1=x] 
= \exp\Big\{
\im axt-\sigma^\alpha |t|^\alpha \Big[
1+\im\beta \tan\frac{\pi\alpha}{2}\sign(t) 
\Big]
\Big\}
\end{align*}
and hence for $1<\alpha<2$,
$\E[X_2\mid X_1=x]=ax$.
The support of spectral measure $\Gamma$ consists of four points in
$\Sphere^1$,
\begin{align*}
\Gamma(0,\pm 1) &= \frac{1}{2} \sigma^\alpha (1 \pm \beta), \\
\Gamma\Big(
\pm \frac{1}{\sqrt{1+a^2}}, \pm \frac{a}{\sqrt{1+a^2}} \Big) 
&= \frac{1}{2}\frac{\sigma^\alpha}{1-|a|^\alpha}
(1+a^2)^{\frac{\alpha}{2}}\Big(
1\pm\beta \frac
{1-|a|^\alpha}{1-\sign(a)|a|^\alpha}
\Big).
\end{align*}
Hence,
\begin{align*}
\int_{\Sphere^1}|s_2-cs_1|^\alpha\, \Gamma(\d s) 
&= \sigma^\alpha \Big(
1+\frac{|a-c|^\alpha}{1-|a|^\alpha}
\Big),\\
\int_{\Sphere^1}\sign(s_2-cs_1)|s_2-cs_1|^\alpha\, \Gamma(\d s) 
&= \beta \sigma^\alpha \Big(
1+\frac{\sign(a-c)|a-c|^\alpha}{1-\sign(a)|a|^\alpha}
\Big).
\end{align*}
Substitution of these relations into \eqref{eq:predict:multi:stable:sigmabeta} yields 
$\sigma_0$ and $\beta_0$ that can be used for calculating the fractional absolute prediction error 
by Proposition \ref{prop:bivariate:stable:condi:mom}.

Another example is the prediction for sub-Gaussian random vector.
Let $0<\alpha<2$, $|\gamma|\le1$, and let $(G_1,\,G_2)$  be zero mean 
Gaussian random vector with covariance matrix
\begin{equation}\label{eq:sigma}
\Sigma= \begin{pmatrix}
 1 & \gamma \\
 \gamma & 1
\end{pmatrix}.
\end{equation}
Let $A$ be a positive $\alpha/2$-stable random variable, given by the LP
transform
\[
\E[\e^{-t A}]= \e^{-t^{\alpha/2}},\quad t>0,
\]
such that it is independent of $(G_1,G_2)$. 
The vector 
$(X_1, X_2) = (A^{1/2}G_1,A^{1/2}G_2)$ is called a sub-Gaussian symmetric
$\alpha$-stable random vector. In \cite{samorodnitsky:taqqu:1991}, $\E[X_2\mid
X_1]=\gamma X_1$ is shown. Since we have the ch.f. 
\begin{align*}
\E[\e^{\im t(X_2-\gamma X_1)}] &= \E\big[\E[\e^{\im t(A^{1/2}G_2-\gamma
A^{1/2}G_1)}]\mid A\big] \\
&= \E\big[\exp\{-\frac{t^2}{2}A(-\gamma,1)\Sigma(-\gamma,1)'\}\big] \\
&= \E\big[\exp\big\{-\frac{t^2(1-\gamma^2)}{2}A\big\}\big] = \exp\left\{-\left(\frac{1-\gamma^2}{2}\right)^{\alpha/2} t^\alpha\right\},
\end{align*}
due to the fractional moment 
\eqref{fractional:mom:sym:stable}, we get the fractional error  
\[
\E\big[|X_2 -\E[X_2\mid X_1] |^{1+\lambda}\big] = \frac{
\lambda
\Gamma(1-\frac{1+\lambda}{\alpha})}{\sin(\frac{\lambda\pi}{2})
\Gamma(1-\lambda)} \Big(
\frac{1-\gamma^2}{2}
\Big)^{\frac{1+\lambda}{2}}, 
\]
where $1 < 1+\lambda < \alpha$.
If $X_2$ is predicted by a linear function $cX_1$, 
in a similar manner, we obtain 
\[
\E[\e^{\im t(X_2-cX_1)}]=
\exp\left\{-\left(\frac{1-2\gamma c+c^2}{2}\right)^{\alpha/2}t^\alpha\right\},
\quad t \in \R,
\]
which yields
\[
\E\big[|X_2 -cX_1 |^{1+\lambda}\big] = \frac{
\lambda
\Gamma(1-\frac{1+\lambda}{\alpha})}{\sin(\frac{\lambda\pi}{2})
\Gamma(1-\lambda)}
\left(\frac{1-2\gamma c+c^2}{2}\right)^\frac{1+\lambda}{2}. 
\]
Alternatively, we could use the spectral measure of sub-Gaussian
random vector given in \cite[Proposition 2.5.8]{Samorodnitsky:Taqqu:1994}. 
Since it is given in a closed form, we obtain the fractional error
directly from ch.f. here.

Next we examine the prediction of the $\alpha$-stable OU process with
$0<\alpha<2$ and $\gamma>0$ given by 
\[
X_t = \e^{-\gamma t} X_0 +\int_0^t \e^{-\gamma(t-s)}\,\d Z_s, \quad t > 0,
\]
where $\{Z_t\}_{t\in\R}$ is the symmetric $\alpha$-stable motion. We set
$X_0=\int_{-\infty}^0 \e^{\gamma s}\,\d Z_s$ to obtain the stationary version,
see \cite[Example 3.6.3]{Samorodnitsky:Taqqu:1994} for its definition. 
Then, the conditional ch.f.~of $X_t$ given $X_0$ is 
\begin{align*}
\varphi_{X_0}(u) = \E[\e^{\im uX_t}\mid X_0] &= \exp\left\{\im u \e^{-\gamma t} X_0\right\}
\exp\left\{-\int_0^t |u\e^{-\gamma (t-s)}|^\alpha\,\d s\right\} \\
&= \exp\left\{ \im u \e^{-\gamma t} X_0 - \frac{1- \e^{-\alpha \gamma t}}{\alpha\gamma
} |u|^\alpha \right\}, 
\end{align*}
which yields $\E[X_t\mid X_0]=\e^{-\gamma t}X_0$ for $\alpha>1$. Since the mean squared
error of the prediction is not available, we use the fractional absolute moment 
of order $1<1+\lambda<\alpha$. More
generally, we measure the error of a linear approximation $cX_0$ with $c \in \R$.  
\begin{proposition}\label{stable:ou:prediction}
Let $X_t$ be an $\alpha$-stable OU-process driven by the symmetric stable
 motion with the location parameter $\delta=0$. Then for $1<1+\lambda
 <\alpha$, 
\[
\E\big[|X_t -c X_0 |^{1+\lambda}\big] = 
\frac{
\lambda
\Gamma(1-\frac{1+\lambda}{\alpha})}{\sin(\frac{\lambda\pi}{2})
\Gamma(1-\lambda)}\Big[
\frac{1-\e^{-\alpha \gamma t}+(c-\e^{-\gamma t})^\alpha}{\alpha\gamma}
\Big]^{\frac{1+\lambda}{\alpha}}
\]
and hence putting $c=\e^{-\gamma t}$, we obtain
\[
\E\big[|X_t -\E[X_t\mid X_0] |^{1+\lambda}\big] = 
\frac{
\lambda
\Gamma(1-\frac{1+\lambda}{\alpha})}{\sin(\frac{\lambda\pi}{2})
\Gamma(1-\lambda)}
\Big(\frac{1-\e^{-\alpha \gamma t} }{ \alpha \gamma
 }
\Big)^{\frac{1+\lambda}{\alpha}}. 
\] 
\end{proposition}
\begin{proof} 
Since 
\[
\E[\e^{\im u(X_t-cX_0)}\mid X_0] = \exp\{\im u(e^{-\gamma t}-c)X_0\}
\exp\big\{-\frac{1}{\alpha\gamma}(1-\e^{-\alpha\gamma t})|u|^\alpha\big\},
\]
$u \in \R$, we may use formula \eqref{eq:mu:general:stable:symm} in Proposition
\ref{prop:stable:moment} with $\mu=(c-\e^{-\gamma t})X_0$
and $\sigma^\alpha=\frac{1}{\alpha\gamma}(1-\e^{-\alpha \gamma t})$.
Consequently,
\begin{align*}
&\E[|X_t-cX_0|^{1+\lambda}\mid X_0] = \frac{\lambda
\sigma^{1+\lambda}}{\sin (\frac{\lambda \pi}{2})\Gamma(1-\lambda)} 
\Big\{
\frac{(c-\e^{-\gamma t})X_0}{\sigma}\int_0^\infty v^{-(1+\lambda)}\e^{-v^\alpha} \\
&\quad \times \sin
\Big[
\frac{(c-\e^{-\gamma t})X_0}{\sigma}v
\Big]\,\d v 
+\alpha \int_0^\infty v^{\alpha-\lambda-2}\e^{-v^\alpha} \cos\Big[
\frac{(c-e^{-\gamma t})X_0}{\sigma}v
\Big]\,\d v
\Big\}.
\end{align*}
After taking expectation w.r.t.~$X_0$ and applying Fubini's theorem,
we get
\begin{align*}
\E[|X_t-cX_0|^{1+\lambda}] &=
\frac{\lambda
\sigma^{1+\lambda}}{\sin (\frac{\lambda \pi}{2})\Gamma(1-\lambda)} 
\Big[
1+ \frac{1}{\alpha\gamma } \Big(
\frac{c-\e^{-\gamma
t}}{\sigma}
\Big)^\alpha
\Big] \\
&\qquad \times \int_0^\infty u^{-\frac{1+\lambda}{\alpha}} \exp\left\{-\Big[
1+\frac{1}{\alpha\gamma }
\Big(
\frac{c-\e^{-\gamma t}}{\sigma}
\Big)^\alpha
\Big]u \right\}\,\d u. 
\end{align*}
Then the result is implied by 
$\sigma^\alpha=\frac{1}{\alpha\gamma}(1-\e^{-\alpha\gamma t})$
and definition of the gamma function. 
\end{proof} 

Since the finite
dimensional distribution of the $\alpha$-stable OU process is a
multivariate stable, we may use Proposition
\ref{prop:bivariate:stable:condi:mom} similarly as before. 
The spectral measure is given in \cite[Example 3.6.4]{samorodnitsky:taqqu:1991}.  

\subsection{Evaluation of conditional expectation related with Linnik law}
$(1)$ Let $(X_1,X_2)$ be a bivariate Linnik distribution with ch.f. 
\begin{align*}
\varphi(t_1,t_2) = \big[1+(\boldsymbol t'\Sigma \boldsymbol t)^{\alpha/2}\big]^{-\beta},
\quad \boldsymbol t'=(t_1,t_2) \in \R^2,
\end{align*}
where $0<\alpha\le 2$, $\beta>0$ and $\Sigma$ is given by \eqref{eq:sigma},
see e.g.~\cite{lim:teo:2010}.
\begin{proposition}
Let $(X_1,X_2)$ be a bivariate Linnik random vector. Then $\E[X_2\mid
X_1]=\gamma X_1$ and for any $c \in \R$ and $1<\alpha<2$ it follows that
\[
 \E[|X_2-cX_1|^{1+\lambda}]=\frac{\lambda \beta}{\sin
 (\frac{\lambda\pi}{2})\Gamma(1-\lambda)} |c^2-2\gamma c+1|^{\frac{1+\lambda}{2}}
B\bigl(1-\frac{1+\lambda}{\alpha},\beta+\frac{1+\lambda}{\alpha}\bigr)
\]
and therefore 
\[
\E[ |X_2-\E[X_2\mid X_1]|^{1+\lambda}] = 
\frac{\lambda \beta}{\sin
 (\frac{\lambda\pi}{2})\Gamma(1-\lambda)} 
 (1-\gamma^2)^{\frac{1+\lambda}{2}}
 B\bigl(1-\frac{1+\lambda}{\alpha},\beta+\frac{1+\lambda}{\alpha}\bigr).
\]
\end{proposition}
\begin{proof}
 To obtain $\E[X_2\mid X_1]$, we use the decomposition by \cite{devroye:1990}
 of univariate Linnik law, which is also applicable in our bivariate case. Let
 $(Y_1,Y_2)$ be a sub-Gaussian random vector with ch.f.
\[
 \varphi(t_1,t_2)= \e^{-(\boldsymbol t' \Sigma \boldsymbol t)^{\alpha/2}},
 \quad \boldsymbol t=(t_1,t_2)'
\]
and let $Z$ be an independent random variable with density
\[
 f(x)=\frac{\e^{-x^{1/\beta}}}{\Gamma(1+\beta)},\quad x>0.
\] 
Then we observe that
 $(X_1,X_2)\eqd(Y_1Z^{1/\alpha\beta},Y_2Z^{1/\alpha\beta})$,
 which leads to
\begin{align*}
 \E[X_2\mid X_1] &\eqd \E\bigl[
\E[Y_2Z^{1/\alpha\beta} \mid Y_1,Z^{1/\alpha\beta}] \mid Y_1Z^{1/\alpha\beta}
\bigr] \\
&= \E\bigl[Z^{1/\alpha\beta}\E[Y_2\mid Y_1]\mid Y_1Z^{1/\alpha\beta}\bigr] 
= \gamma Y_1 Z^{1/\alpha\beta} \eqd \gamma X_1,
\end{align*}
where the conditional expectation of the sub-Gaussian random vector is
 used. Now put $t_1=-cu$ and $t_2=u$ in $\varphi(t_1,t_2)$ to obtain
\[
 \E[\e^{\im u(X_2-cX_1)}]=\left[1+(c^2-2\gamma
 c+1)^{\alpha/2}|u|^\alpha\right]^{-\beta}
\]
and we conclude our result from Subsection \ref{subsec:linnik}.
\end{proof}
\noindent
$(2)$ Let $Z$ be a symmetric stable random variable with exponent $0<\alpha<2$ and $E$ be the standard exponential 
random variable, i.e.
\[
 \E[\e^{\im tZ}]= \e^{\im t\delta-\sigma^\alpha |t|^\alpha}\quad \mathrm{and}\quad
 \E[\e^{\im tE}]=(1-\im t)^{-1}, \qquad t \in \R,\quad \delta \in\R,\ \sigma>0.
\]
We consider a bivariate distribution $(X_1,X_2)\stackrel{d}{=}
(E^{1/\alpha}Z,E)$ as in \cite{kozubowski:meerschaert:2009},
where $Z$ is a stable subordinator, which yields a bivariate
distribution with exponential and Mittag-Leffler marginals. 
Note that the marginal $X_1$ has no second moment, whereas
$X_2$ has any power moments.
Since the
conditional ch.f. of $X_1$ given $X_2$ is
\[
 \varphi_{X_2}(t)=\E[\e^{\im tE^{1/\alpha}Z}\mid
 X_2]=\e^{\im t \delta X_2^{1/\alpha}-\sigma^\alpha X_2 |t|^\alpha},
\]
the conditional expectation has the form $\E[X_1\mid X_2]=\delta
X_2^{1/\alpha}$. If we predict $X_1$ by $cX_2^{1/\alpha}$ with a
constant $c$, 
the following result holds. 
\begin{proposition}
Let $(X_1,X_2) \stackrel{d}{=}(E^{1/\alpha}Z,E)$ be a bivariate random
vector such that $Z$ is a symmetric stable with location
$\delta\in\R$ and scale $\sigma>0$ and $E$ is the standard exponential.
Then $\E[X_1\mid
X_2]=\delta X_2^{1/\alpha}$ and for any $c \in \R$ and $1<\alpha<2$ it follows that
\begin{align*}
 \E[|X_1-cX_2^{1/\alpha}|^{1+\lambda}] = \frac{\lambda
 \sigma^{1+\lambda}\Gamma(1+\frac{1+\lambda}{\alpha})}{\sin (\frac{\lambda\pi}{2})\Gamma(1-\lambda)}
& \Big[
\frac{\delta-c}{\sigma}\int_0^\infty u^{-(1+\lambda)}\e^{-u^\alpha} \sin
 \Big(\frac{\delta-c}{\sigma}u\Big)\,\d u\\
& + \alpha \int_0^\infty u^{\alpha-\lambda-2}\e^{-u^\alpha}\cos \Big(
\frac{\delta-c}{\sigma}u
\Big)\,\d u
\Big]
\end{align*}
and therefore 
\[
\E[ |X_1-\E[X_1\mid X_2]|^{1+\lambda}] = 
\frac{\lambda \sigma^{1+\lambda}\Gamma(1+\frac{1+\lambda}{\alpha}) \Gamma(1-\frac{1+\lambda}{\alpha})}{\sin
 (\frac{\lambda\pi}{2})\Gamma(1-\lambda)}.
\]
\end{proposition}
\begin{proof}
 Since $\E [\e^{\im t(X_1-cX_2^{1/\alpha})}\mid X_2]=\e^{\im t (\delta-c)
 X_2^{1/\alpha}-\sigma^\alpha X_2 |t|^\alpha}$ is ch.f. of symmetric
 stable distribution, we apply \eqref{eq:mu:general:stable:symm} of Proposition
 \ref{prop:stable:moment} and then take expectation with respect to
 $X_2$, which is justified by Fubini's theorem.
\end{proof}
\noindent

\subsection{Estimation errors of regression model} 
We consider the basic regression model
\[
Y_j=\theta_0+x_j\theta_1+\varepsilon_j,\quad j=1,2,\ldots,n,
\]
where $(\varepsilon_j)$ is an iid 
sequence of symmetric random variables. 
It is well-known that the least squares estimator $(\widehat \theta_0,\,\widehat \theta_1)$, 
which is the best linear unbiased estimator if the $\varepsilon_j$ follow Gaussian distribution, 
has the form
\[
\widehat\theta_0 = {\overline Y} - \widehat\theta_1 {\overline x}, \qquad \widehat\theta_1 =
\frac{\sum_{j=1}^n (x_j-\overline x)(Y_j-\overline Y)}{\sum_{j=1}^n (x_j-\overline x)^2},
\]
where $\overline x = \frac{1}{n}\sum_{j=1}^n x_j$ and $\overline Y = \frac{1}{n}\sum_{j=1}^n Y_j$.
For our purpose, it will be convenient to rewrite this as
\begin{align*}
\widehat \theta_0 &= \theta_0 -\sum_{i=1}^n \frac{(x_i-\overline x)\overline x -
\sum_{j=1}^n(x_j-\overline x)^2/n}{\sum_{j=1}^n (x_j-\overline x)^2}\varepsilon_i, \\
\widehat \theta_1 &= \theta_1 + \sum_{i=1}^n \frac{x_i-\overline x}{\sum_{j=1}^n (x_j-\overline x)^2}
\varepsilon_i. 
\end{align*}
We express the fractional errors for the case of $\alpha$-stable noise distribution.
In a similar manner, it would be possible to calculate the fractional errors
for other regression-type estimators, e.g.~\cite{blattberg:sargent:1971}, and compare 
the goodness of estimators.   

If $\varepsilon_1$ is a standard symmetric stable random variable with
parameters $\delta=0$, $\sigma=1$ and $\alpha>1$, then
the characteristic functions of estimation errors are 
\begin{align*}
\E[\e^{\im t(\widehat \theta_k-\theta_k)}] &= \e^{-\sigma_k^\alpha |t|^\alpha},
\quad t \in \R, \quad k=0,1, 
\end{align*}
where
\[
\sigma_0^\alpha= \sum_{i=1}^n \left(
\frac{|(x_i-\overline x)\overline x - \sum_{j=1}^n(x_j-\overline x)^2/n|}{\sum_{j=1}^n (x_i-\overline x)^2}
\right)^\alpha
\quad \mathrm{and}\quad 
\sigma_1^\alpha = \sum_{i=1}^n
\left(\frac{|x_i-\overline x|}{\sum_{j=1}^n (x_j-\overline x)^2}\right)^\alpha. 
\]
This together with \eqref{fractional:mom:sym:stable} yields, for $1 < 1+\lambda < \alpha$,
\begin{align*}
\E[|\widehat \theta_k-\theta_k|^{1+\lambda}] &= \frac{\lambda
\Gamma(1-\frac{1+\lambda}{\alpha})}{\sin(\frac{\lambda\pi}{2})
\Gamma(1-\lambda)}\,\sigma_k^{1+\lambda}, \quad k=0,1.
\end{align*}
Interestingly, if $\boldsymbol
\varepsilon=(\varepsilon_1,\varepsilon_2,\ldots,\varepsilon_n)$ is
an elliptically contoured stable random vector with ch.f.
$
\E[\e^{\im {\boldsymbol t' \boldsymbol\varepsilon}}]= \e^{-|\boldsymbol t'
\boldsymbol I
\boldsymbol t|^{\alpha/2}}$ for $\boldsymbol t' \in \R^n$, where 
$\boldsymbol I$ is $n\times n$ identity matrix, then we obtain
closer results to Gaussian case. Namely, for $1<1+\lambda<\alpha$,
\begin{align*}
\E[|\widehat \theta_k-\theta_0|^{1+\lambda}] &=  \frac{\lambda
\Gamma(1-\frac{1+\lambda}{\alpha})}{\sin(\frac{\lambda\pi}{2})
\Gamma(1-\lambda)}\,\overline \sigma_k^{1+\lambda},\quad k=0,1,
\end{align*}
where 
\[
 \overline \sigma_0= \Big(
\frac{\overline x^2}{\sum_{j=1}^n
(x_j-\overline x)^2 }+\frac{1}{n}
\Big)^{1/2}\quad \mathrm{and}\quad \overline 
\sigma_1=\frac{1}{\left[\sum_{j=1}^n
(x_j-\overline x)^2\right]^{1/2}}.
\]

Moreover, if $\boldsymbol
\varepsilon=(\varepsilon_1,\varepsilon_2,\ldots,\varepsilon_n)$ is
a multivariate Linnik random vector with ch.f.
$
\E[\e^{\im {\boldsymbol t' \boldsymbol\varepsilon}}]= \{ 1+ (\boldsymbol t' 
 \boldsymbol I \boldsymbol t)^{\alpha/2} \}^{-\beta} $ for $\boldsymbol
 t' \in \R^n$, we obtain in a similar manner that ($1 < 1+\lambda < \alpha$)
\[
\E[|\widehat \theta_k-\theta_k|^{1+\lambda}] = 
\frac{\lambda \beta
B(1-\frac{1+\lambda}{\alpha},\beta+\frac{1+\lambda}{\alpha})}{\sin
(\frac{\lambda \pi}{2})\Gamma(1-\lambda)}\, \overline \sigma_k^{1+\lambda},\quad k=0,1.
\]


\appendix{Proof of Lemma \ref{lem:condition:moment:id}}
\begin{proof}
We express ch.f.~$\varphi(t)$, given by \eqref{def:id-ch-f},
as the product $\varphi_1(t)\cdot\varphi_2(t)$, where 
\begin{align*}
\varphi_1(t) &:= \exp\Big\{\im \delta t +
\int_{|x|\le 1}(\e^{\im tx}-1-\im tx)\,\nu(\d x)
\Big\}, \quad t \in \R,\\
\varphi_2(t) &:= \exp\Big\{
\int_{|x|>1}(\e^{\im tx}-1)\,\nu(\d x)
\Big\}, \quad t \in \R. 
\end{align*}
Since the distribution with ch.f.~$\varphi_1(t)$ has moments of any positive order, it
suffices to consider $\varphi_2(t)$. We use the necessary and sufficient
condition \eqref{lem:fractional:ch:kawata:condi} for the existence of $m_{1+\lambda}$,
see Lemma \ref{lem:fractional:ch}. The following inequalities
\begin{align*}
\frac{a}{1+a} \le 1-\e^{-a} \le a,& \qquad a\ge 0,\\
1-\cos b \le \frac{b^2}{2} \le \frac{b}{2},&\qquad 0\le b \le 1, 
\end{align*}
and the fact
\[
 \int_{|x|>1} (1-\cos tx)\,\nu(\d x) \le \int_{|x|>1}\Big(
\frac{(tx)^2}{2}\wedge 1
\Big)\,\nu(\d x)=: c_t<\infty
\]
are used to obtain
\begin{align*}
1-\Re\varphi_2(t) &\ge 1-\exp\Big\{\int_{|x|>1}(\cos tx -1)\,\nu(\d x)\Big\} \ge
 \frac{1}{1+c_t}\int_{|x|>1}(1-\cos tx)\,\nu(\d x), \\
1-\Re\varphi_2(t) &= 1-\exp\Big\{\int_{|x|>1}(\cos tx -1)\,\nu(\d x)\Big\} +
 \exp\Big\{\int_{|x|>1}(\cos tx -1)\,\nu(\d x)\Big\} \\
&\quad -\cos\Big(\int_{|x|>1}\sin tx \,\nu(\d x)\Big)
\exp\Big\{\int_{|x|>1}(\cos tx -1)\,\nu(\d x)\Big\} \\
&\le  \int_{|x|>1}(1-\cos tx )\,\nu(\d x) +\frac{1}{2}\int_{|x|>1}|\sin
 tx| \,\nu(\d x).
\end{align*}
Then we notice by Fubini's theorem that
\begin{align*}
\int_0^\infty t^{-(2+\lambda)} \int_{|x|>1} (1-\cos tx)\,\nu(\d x)\,\d t &=
 \int_0^\infty \frac{1-\cos v}{v^{2+\lambda}}\,\d v \int_{|x|>1}|x|^{1+\lambda}
\,\nu(\d x), \\
\int_0^\infty t^{-(2+\lambda)} \int_{|x|>1}|\sin tx|\,\nu(\d x)\,\d t &= 
\int_0^\infty \frac{|\sin v|}{v^{2+\lambda}}\,\d v \int_{|x|>1}|x|^{1+\lambda}
\,\nu(\d x). 
\end{align*}
Since double integrals on the left-hand sides exist if and only if the integrals on the right-hand sides exist, 
the condition \eqref{lem:fractional:ch:kawata:condi} is equivalent to the existence of 
$\int_{|x|>1}|x|^{1+\lambda}\,\nu(\d x)$. 
\end{proof}


\begin{thebibliography}{99}                  
%

\bibitem{adler:feldman:taqqu:1998}
\n{R. J. Adler, R. E. Feldman \and M. S. Taqqu (Eds.)}
{\it A Practical Guide to Heavy Tails: Statistical Techniques and Applications,}
 Birkh\"auser, New York, 1998.

\bibitem{blattberg:sargent:1971}
\n{R. Blattberg \and T. Sargent}, 
{\it Regression with non-Gaussian stable disturbances: some sampling results,}
Econometrica 39 (1971), pp. 501--510. 

\bibitem{brown:1970}
\n{B. M. Brown}, 
{\it Characteristic functions, moments, and the central limit theorem,} 
Ann. Math. Statist. 41 (1970), pp. 658--664.


\bibitem{brown:1972}
\n{B. M. Brown}, 
{\it Formulae for absolute moments,} 
J. Austral. Math. Soc. 13 (1972), pp. 104--106.


\bibitem{cline:brockwell:1985}
\n{D. B. H. Cline \and P. J. Brockwell},
{\it Linear prediction of ARMA processes with infinite variance,}
Stochastic Process. Appl. 19 (1985), pp. 281--296. 

\bibitem{cressie:borkent:1986}
\n{N. Cressie \and M. Borkent},
{\it The moment generating function has its moments,}
J. Statist. Plann. Inference 13 (1986), pp. 337--344. 

\bibitem{cressie:davis:folks:policello:1981}
\n{N. Cressie, A. S. Davis, J. L. Folks \and G. E. Policello}, 
{\it The moment-generating function and negative integer moments,}
Amer. Statist. 35 (1981), pp. 148--150.

\bibitem{devroye:1990}
\n{L. Devroye}, 
{\it A note on Linnik's distribution,} 
Statist. Probab. Lett. 9 (1990), pp. 305--306.

\bibitem{Gradshteyn:Ryzhik:2007}
\n{I. S. Gradshteyn \and I. M. Ryzhik}, 
{\it Table of Integrals, Series, and Products, 7th ed.,}
Academic Press, San Diego, 2007.

\bibitem{hardin:samorodnitsky:taqqu:1991}
\n{C. D. Hardin Jr, G. Samorodnitsky \and M. S. Taqqu},
{\it Nonlinear regression of stable random variables,}
Ann. Appl. Probab. 1 (1991), pp. 582--612. 

\bibitem{hsu:1951}
\n{P. L. Hsu},
{\it Absolute moments and characteristic functions,}
J. Chinese Math. Soc. 1 (1951), pp. 259--280.


\bibitem{kawata:1972}
\n{T. Kawata},
{\it Fourier Analysis in Probability Theory,}
Academic Press, New York, 1972.

\bibitem{kokoszka:1996}
\n{P. S. Kokoszka}
{\it Prediction of inifinite variance fractional ARIMA,}
Probab. Math. Statist. 16 (1996), pp. 65--83.

\bibitem{kozubowski:meerschaert:2009}
\n{T. J. Kozubowski \and M. M. Meerschaert}, 
{\it A bivariate infinitely divisible distribution with exponential and Mittag-Leffler marginals,}
Statist. Probab. Lett. 79 (2009), pp. 1596--1601.  

\bibitem{kozubowski:podgorski:samorodnitsky:1999}
\n{T. J. Kozubowski, K. Podg\'orski, \and G. Samorodnitsky}, 
{\it Tails of L\'evy measure of geometric stable random variables,}
Extremes 1 (1999), pp. 367--378. 

\bibitem{laue:1980}
\n{G. Laue},
{\it Remarks on the relation between fractional moments and fractional derivatives of characteristic functions,}
J. Appl. Probab. 17 (1980), pp. 456--466. 

\bibitem{laue:1986}
\n{G. Laue},
{\it Results on moments of non-negative random variables,}
Sankhy\=a Ser. A 48 (1986), pp. 299--314.

\bibitem{lim:teo:2010}
\n{S. C. Lim \and L. P. Teo}, 
{\it Analytic and asymptotic properties of multivariate generalized Linnik's probability densities,} 
J. Fourier Anal. Appl. 16 (2010), pp. 715--747. 

\bibitem{lin:1998}
\n{G. D. Lin}, 
{\it On the Mittag-Leffler distributions,} 
J. Statist. Plann. Inference 74 (1998), pp. 1--9. 

\bibitem{linnik:1953}
\n{Yu. V. Linnik}, 
{\it Linear forms and statistical criteria I, II,}
Ukra\"\i n. Mat. Zh 5 (1953), pp. 207--243 and pp. 247--290.
[English translation in Selected Trans. Math. Statist. Probab. 
3 (1962), pp. 1--90, American Mathematical Society, Providence.]

\bibitem{matsui:mikosch:2010}
\n{M. Matsui \and T. Mikosch}, {\it Prediction in a Poisson cluster model,} 
J. Appl. Probab. 47 (2010), pp. 350--366. 

\bibitem{nguyen:1995}
\n{T. T. Nguyen}, 
{\it Conditional distributions and characterizations of multivariate stable distribution,}
J. Multivariate Anal. 53 (1995), pp. 181--193. 

\bibitem{paolella:2007}
\n{M. S. Paolella}, 
{\it Intermediate Probability: A Computational Approach,}
John Wiley \& Sons, Chichester, 2007.

\bibitem{pinelis:2011}
\n{I. Pinelis}, 
{\it Positive-part moments via the Fourier-Laplace transform,} 
J. Theoret. Probab. 24 (2011), pp. 409--421. 

\bibitem{podlubny:1999}
\n{I. Podlubny}, 
{\it Fractional Differential Equations,}
Academic Press, San Diego, 1999.

\bibitem{ramachandran:1969}
\n{R. Ramachandran}, 
{\it On characteristic functions and moments,} 
Sankhy\=a Ser. A 31 (1969), pp. 1--12.

\bibitem{samko:1993}
\n{S. G. Samko, A. A. Kilbas \and O. I. Marichev}, 
{\it Fractional Integrals and Derivatives: Theory and Applications,}
Gordon and Breach Science Publishers, Yverdon, 1993.

\bibitem{samorodnitsky:taqqu:1991}
\n{G. Samorodnitsky \and M. S. Taqqu}, 
{\it Conditional moments and linear regression for stable random variables,}
Stochastic Process. Appl. 39 (1991), pp. 183--199.

\bibitem{Samorodnitsky:Taqqu:1994}
\n{G. Samorodnitsky \and M. S. Taqqu}, 
{\it Stable Non-Gaussian Random Processes: Stochastic Models with Infinite Variance,} 
Chapman \& Hall/CRC, Boca Raton, 1994.

\bibitem{sato:1999}
\n{K.-I. Sato},
{\it L\'evy Processes and Infinitely Divisible Distributions,}
Cambridge University Press, Cambridge, 1999.

\bibitem{shanbhag:sreehari:1977}
\n{D. N. Shanbhag \and M. Sreehari},  
{\it On certain self-decomposable distributions,} 
Zeit. Wahrsch. Verw. Gebiete 38 (1977), pp. 217--222.

\bibitem{vonBahr:1965}
\n{B. von Bahr},
{\it On the convergence of moments in the central limit theorem,}
Ann. Math. Statist. 36 (1965), pp. 808--818.

\bibitem{wolfe:1971}
\n{S. J. Wolfe}, 
{\it On moments of infinitely divisible distribution functions,}
Ann. Math. Statist. 42 (1971), pp. 2036--2043. 

\bibitem{wolfe:1973}
\n{S. J. Wolfe}, 
{\it On the local behavior of characteristic functions,} 
Ann. Probab. 1 (1973), pp. 862--866.

\bibitem{wolfe:1975a}
\n{S. J. Wolfe}, 
{\it On moments of probability distribution functions,} 
In: Fractional Calculus and Its Applications, B.~Ross (ed.), 
Lect. Notes in Math. 457, Springer, Berlin, 1975, pp. 306--316.

\bibitem{wolfe:1975b}
\n{S. J. Wolfe}, 
{\it On derivatives of characteristic functions,}
Ann. Probab. 3 (1975), pp. 737--738.

\bibitem{wolfe:1978}
\n{S. J. Wolfe}, 
{\it On the behavior of characteristic functions on the real line,}
Ann. Probab. 6 (1978), pp. 554--562.

\bibitem{zolotarev:1986}
\n{V. M. Zolotarev},
{\it One-dimensional Stable Distributions,}
Transl. Math. Monogr. 65, American Mathematical Society, Providence, 1986.

\end{thebibliography}
\end{document}